\documentclass[12pt,a4paper,leqno]{article}
\usepackage[cp1250]{inputenc}
\usepackage[a4paper, total={6in, 8in}]{geometry}
\usepackage{amsmath}
\usepackage{amsthm}
\usepackage{amssymb}
\usepackage{enumerate}
\usepackage[colorlinks]{hyperref}
\hypersetup{ hidelinks = true}

\usepackage{url}
\usepackage{breakurl}

 \expandafter\def\expandafter\UrlBreaks\expandafter{\UrlBreaks
    \do\a\do\b\do\c\do\d\do\e\do\f\do\g\do\h\do\i\do\j
    \do\k\do\l\do\m\do\n\do\o\do\p\do\q\do\r\do\s\do\t
    \do\u\do\v\do\w\do\x\do\y\do\z\do\A\do\B\do\C\do\D
    \do\E\do\F\do\G\do\H\do\I\do\J\do\K\do\L\do\M\do\N
    \do\O\do\P\do\Q\do\R\do\S\do\T\do\U\do\V\do\W\do\X
    \do\Y\do\Z\do\/\do-}

\DeclareMathAlphabet{\mathpzc}{OT1}{pzc}{m}{it}

\newcommand{\bbR}{\mathbb{R}}

\newcommand{\calF}{\mathcal{F}}
\newcommand{\calG}{\mathcal{G}}

\newcommand{\xgeq}{\geqslant}
\newcommand{\xleq}{\leqslant}

\newtheorem*{thm}{Theorem}
\newtheorem*{lem}{Lemma}
\newtheorem*{prob}{Problem}

\title{Intersections of vertical lines with non-vertical lines}
\author{Mateusz Kula}
\date{}

\begin{document}
\maketitle
\begin{abstract}
  We answer negatively a question: if $\calF$ is a family of $n\xgeq 3$ non-vertical, pairwise non-parallel lines  on the plane and $\bigcap \calF=\emptyset$, is there a vertical line $L$ such that $L\cap\bigcup \calF$ has exactly $n-1$ or $n-2$ points? 
\end{abstract}

  In \cite{leo} P. Leonetti posed the following problem.
  Let $\calF$ be a family of $n\xgeq 3$ non-vertical, pairwise non-parallel lines  on the plane such that they are not concurrent. 
  Is there a vertical line $L$ such that $L\cap\bigcup \calF$ has exactly $n-1$ or $n-2$ points?
  We answer this question negatively, giving  appropriate examples for $n\xgeq7$.

For our purposes, consider a bijection $f$ between the plane $\bbR^2$ and the family of all non-vertical lines such that  $$f(A,B):=\{(x,y)\in\bbR^2\colon y+Ax+B=0\}.$$
	\begin{lem}
	$(P,Q)\in f(A,B)$ if and only if $(A,B)\in f(P,Q)$.
	\end{lem}
	\begin{proof}
	Both conditions are equivalent to the equality $Q+AP+B=0$.
	\end{proof}
  Let $q_1,\ldots,q_7\in\bbR^2$ be the vertices of a regular heptagon such that no two of them lie on one vertical line.
  If vertices $q_i, q_j$ lie on the line $K$, then there exist exactly $4$ lines parallel to $K$ such that each of these lines contains at least one vertex $q$.
  The family
  $$\calF:=\{f(q_i)\colon 1\xleq i\xleq 7\}$$
  consists of $n=7$ pairwise non-parallel lines, which are not concurrent. Indeed, if there was a point $p$ such that $p\in f(q_i)$ for each $i$, then by the Lemma points $q_1,\ldots, q_7$ would all lie on one line $f(p)$, which contradicts the fact that $q_1,\ldots,q_7$ are vertices of the heptagon.\\ \indent
	Suppose a vertical line $L$  intersects the union $\bigcup\calF$ at $k$ points $(A,D_1),\ldots,(A,D_k)$, where $1\xleq k\xleq 7$. 
 Observe that $p\in L\cap\bigcup\calF$ if and only if the line $f(p)$ is parallel to $f(A,D_1)$ and contains at least one vertex $q_i$. Since there are exactly $4$ or $7$ such lines  and $f$ is a bijection, we have $k=4$ or $k=7$.
Hence no vertical line $L$ intersects the union $\bigcup\calF$ in exactly $n-1=6$ or $n-2=5$ points. This answers negatively Leonetti's question.

In our opinion one should consider the following problem.
\begin{prob}
  Assume that  $Q\subseteq \bbR^2$ is a finite subset and $\calG$ is a family of parallel lines which covers $Q$, and such that each line $K\in\calG$ intersects $Q$.
  What are possible values of $|\calG|$?
\end{prob}
Let $I(Q)$ denote  the set of all values $|\calG|$, where $\calG$ is as above.
If $q_1,\ldots,q_n$ are vertices of a regular polygon, then one readily verifies that:
\begin{enumerate}
\item[$(1)$] If $n=2k$, then $I(q_1,\ldots,q_n)=\{k,k+1, 2k\}$;
\item[$(2)$] If $n=2k+1$, then $I(q_1,\ldots,q_n)=\{k, 2k+1\}$;
\item[$(3)$] If $n=4k+2$ and $q$ is the centre of the circle on which points $q_i$ lie, then $I(q,q_1,\ldots,q_n)=\{2k+1,2k+3,4k+3\}$;
\item[$(4)$] If $n=4k$ and $q$ is the centre of the circle on which points $q_i$ lie, then $I(q,q_1,\ldots,q_n)=\{2k+1,4k+1\}$.
\end{enumerate}
\begin{thm}
For any $n\xgeq 7$ there exists a family $\calF$ of $n$ non-vertical, pairwise non-parallel lines, which are not concurrent, such that no vertical line $L$ contains  exactly $n-1$ or $n-2$ points of $\bigcup \calF$.
\end{thm}
\begin{proof}
If $Q\subseteq\bbR^2$ is a set of points as in (1)--(4) and no two points of $Q$ lie on one vertical line, then $f[Q]$ is a desired family of lines. Indeed, with the Lemma one can check that $k\in I(Q)$ if and only if there exists a vertical line, which has exactly $k$ intersections with the union of the family of lines $f[Q]$.
\end{proof}

In the paper \cite[Theorem 2.1]{pin} the set $I(Q)$ was examined, too. For example, R. Pinchasi shows that for any finite subset $Q\subseteq \bbR^2$ of $n$ non-collinear points
$$\max \left(I(Q)\setminus\{n\}\right) \xgeq \left\lfloor \frac{n+1}{2}\right\rfloor.$$

\end{document}